\documentclass{article}
\usepackage[centertags]{amsmath}
\usepackage{amstext,amssymb,amsopn,amsthm}

\newcommand{\Rd}{{\mathbb{R}^d}}
\newcommand{\RR}{{\mathbb{R}}}

\newcommand{\R}{\mathbb{R}}

\newtheorem{thm}{Theorem}

\newtheorem{lem}{Lemma}
\newtheorem{cor}[lem]{Corollary}

\title{Heat kernel of fractional Laplacian
in cones}
\author
{Krzysztof Bogdan \and Tomasz Grzywny
}
\date{November 2008}

\begin{document}
\maketitle
\sloppy \footnotetext{
\\ Research partially supported by KBN (MNiI 1 P03A 026 29) }

This paper is devoted to the memory of Professor Andrzej Hulanicki.
\begin{abstract}
We give sharp estimates for the transition density of the isotropic stable
L\'evy process killed when leaving a right circular cone.
\end{abstract}
\maketitle

\section{Introduction}\label{s:I}
Explicit {sharp} estimates for the Green function of
the Laplacian in $C^{1,1}$ domains were given in 1986 by Zhao 
\cite{MR842803} (see also \cite{MR657523, MR902422}).
Sharp estimates of the Green function of Lipschitz domains 
were given in 2000 by Bogdan \cite{MR1741527}.
Explicit {qualitatively} sharp estimates for the classical heat
kernel in $C^{1,1}$ domains were established in 2002 by Zhang 
\cite{MR1900329} (see also \cite{MR1974093, MR879702}, 
and \cite{MR2255354, MR1968386} for further extensions).
Qualitatively sharp heat kernel estimates in Lipschitz domain
were given in 2003 by Varopulous \cite{MR1969798}.

The development of the boundary potential theory of the fractional
Laplacian follows an analogous path. Green function estimates
were obtained in 1997 and 1998 by Kulczycki and Chen and Song for
$C^{1,1}$ domains \cite{MR1490808, MR1654824} (see
\cite[Corollary 1.8]{CKS2008} for the case of
dimension one, see also \cite{MR1825645}), 
and in 2002 by Jakubowski for Lipschitz domains \cite{MR1991120}
(see also \cite{MR2213639, MR2182071}). 
In 2008 Chen, Kim and Song gave 
a sharp and explicit estimate for the heat kernel of the fractional Laplacian
on $C^{1,1}$ domains \cite{CKS2008}, see (\ref{eq:CKS}) below.
In this note we give an extension of the estimate to the
right circular cones. We also {conjecture}, in agreement with the
results \cite{MR1969798}, a likely
form of the estimate for a more general class of domains:
  \begin{equation}
    \label{eq:*g}
    p_t^D(x,y)\approx P_x(\tau_D>t)P_y(\tau_D>t)p_t(x,y)\,.
  \end{equation}
Here $p_t(x,y)$ is the heat kernel of the fractional Laplacian on the
whole space $\Rd$, and 
$P_x(\tau_D>t)=\int_\Rd p_t^D(x,y)dy$ is the {\it survival
  probability} of the corresponding isotropic $\alpha$-stable L\'evy process. The main result
of this paper, Theorem~\ref{prop:ppu}, asserts that (\ref{eq:*g}) 
holds indeed for the right circular cones for all $t>0$, $x,y\in \Rd$
(see also (\ref{eq:etdG2}) for a more explicit statement).
Noteworthy, all the above-mentioned estimates for bounded $C^{1,1}$
domains have the same form as for the ball
 (in this connection compare \cite[Corollary 1.2]{CKS2008} with
\cite[Corollary 3]{MR1825645}; see also (\ref{eq:CKS}) below).
We also like to note that the right circular cones are merely special Lipschitz domains,
but a number of techniques and explicit formulas make them
an interesting and important test case (see \cite{MR1465162, MR863716,
MR1062058, MR0474525, MR799436}).
We hope to encourage a further study of Lipschitz and more general domains for
stable and other jump-type processes \cite{MR1942325, MR2238879,
  MR2085428, MR2006232}.
We should emphasize that generally the estimates for Lipschitz domains cannot be as
explicit as those for $C^{1,1}$ domains. For instance, the decay rate of
harmonic and parabolic functions in the vertex of a cone delicately depends on the
aperture of the cone, see 
\cite{MR2075671, MR2213639} (see also \cite{MR1741527}). 
Nevertheless Lipschitz domains offer a natural
setting for studying the boundary behavior of the Green function and
the heat kernel of the Brownian motion and $\alpha$-stable L\'evy
processes ($0<\alpha<2$).
This is so because of {\it scaling}, the rich range of
asymptotic behaviors depending on the local geometry of the domain's boundary, connections to
the boundary Harnack principle, approximate factorization of the Green
function, and applications in the perturbation
theory of generators, in particular via the 3G Theorem,
\cite{MR1741527, MR2075671, MR842803, MR2153910, MR2365478, MR2160104, MR2207878,
  MR1671973, MR2283957, 2008KBTJ}, and 3P Theorem \cite{MR2283957}.
Noteworthy, (\ref{eq:*g}) is an approximate factorization of the
heat kernel (see \cite{MR1741527, MR2365478} in this connection).

Cones are also examples of unbounded domains, which are only partially resolved by 
the results of \cite{CKS2008, CKS2008censored} (note that (\ref{eq:CKS}) is
valid only for bounded times). 
We should note that the {upper} bound in (\ref{eq:CKS}) was
proved in 2006 by Siudeja for semibounded convex domains \cite[Theorem
1.6]{MR2255353} (stated for general convex domain in \cite[Remark
  1.7]{MR2255353}). 
It appears that the impulse for the proof of (\ref{eq:CKS})
was given by Siudeja and Kulczycki in \cite[Theorem 4.2]{MR2231884},
see also \cite[Proposition 2.9]{MR2438694} by Kulczycki and Ba\~nuelos.
A similar but weaker upper bound was earlier given in   
\cite[(26)]{MR2075671}, see also \cite{MR2280260, 2008-MK, MR0208681}.
We also remark that \cite[Theorem 4.4]{MR2386098} gives a sharp explicit estimate for
the survival probability of the relativistic process in a half-line.
Generally, the subject is far from exhausted--and it seems manageable
with the existing techniques. 

For completeness we like to mention recent estimates
\cite{MR2357678,MR2320691, MR2238934, MR2286060, MR2365348, MR2430977,2008-MB-RB-TK}
for transition density and potential kernel of jump-type processes.
We need to point out that generally these are {estimates} for processes without killing.
Killing is a  dramatic ``perturbation''
analogous to Schr\"odinger perturbations with
singular negative potentials \cite{MR1825645,2008KBTJ, MR0264757, MR2006232}, and it
strongly influences the asymptotics of the transition density and Green function.
The asymptotics is crucial for solving the Dirichlet problem
for the corresponding
operators, see also \cite{MR2417435, MR2386098}.
As we shall see, the heat kernel of the fractional Laplacian in the right
circular cones has a
power-type asymptotics at infinity, and it decays like the distance to the
boundary to the power $\alpha/2$ except at the vertex, where it decays
with the rate of $\beta\in (0,\alpha)$.

The paper is composed as follows. Below in this section we recall
basic facts about the transition density of the $\alpha$-stable
L\'evy processes killed when first leaving a domain. In
Section~\ref{s:s} we give a sharp explicit estimate for the {
survival probability} $P_x(\tau_D>t)$ for $C^{1,1}$ domains $D$. In
Section~\ref{s:c} we prove our main estimates,
Theorem~\ref{prop:ppu} and (\ref{eq:etdG2}), 
by using the ideas and
results of \cite{CKS2008} and \cite{MR2075671}. Our general
references to the {boundary potential theory} of the fractional
Laplacian are \cite{MR1671973} and \cite{MR2365478}. We also refer
the reader to \cite{KB-TB-MR-TK-RS-ZV} for a broad non-technical
overview of the goals and methods of the theory.

In what follows, $\Rd$  denotes the Euclidean space of dimension $d\ge 1$,
$dx$ is the Lebesgue measure on $\Rd$, and $0<\alpha<2$.
For $t>0$ we let $p_t$ be the smooth real-valued
function on $\Rd$ with the following Fourier transform,
\begin{equation}
  \label{eq:dpt}
  \int_ \Rd p_t(x)e^{ix\cdot\xi}\,dx=e^{-t|\xi|^\alpha}\,,\quad \xi\in
   \Rd\,.
\end{equation}
For instance, $\alpha=1$ yields
\begin{equation*}
p_t(x)=
\Gamma((d+1)/2)\pi^{-(d+1)/2}
\frac{t}{\big(|x|^2+t^2)^{(d+1)/2}}\,,
\end{equation*}
the Cauchy convolution semigroup of functions \cite{MR0290095}.
We generally have that
\begin{equation}
  \label{eq:sca}
  p_t(x)=t^{-d/\alpha}p_1(t^{-1/ \alpha}x)\,,\quad x\in \Rd\,,\;t>0\,.
\end{equation}
This follows from (\ref{eq:dpt}).
The semigroup $P_t f(x)=\int_ \Rd f(y) p_t(y-x) dy$
has $\Delta^{\alpha/2}$ as infinitesimal generator (\cite{MR0481057}, \cite{MR1336382},
\cite{MR1671973}, \cite{MR1873235}), where
\begin{eqnarray*}
   \Delta^{\alpha/2}\varphi(x) &=& 
\frac{2^{\alpha}\Gamma((d+\alpha)/2)}{\pi^{d/2}|\Gamma(-\alpha/2)|}
\lim_{\varepsilon \downarrow 0}\int_{\{|y|>\varepsilon\}}
   \frac{\varphi(x+y)-\varphi(x)}{|y|^{d+\alpha}}dy\,,
\quad
x\in \Rd\,. 
   \end{eqnarray*}
Here $\phi\in C^\infty_c(\Rd)$, i.e. $\phi:\Rd\to\R$ is
smooth and compactly supported on
$\Rd$. Put differently,
\begin{equation*}
\int\limits_{s}^\infty\int\limits_{ \Rd}
p_{u-s}(z-x)\left[
\partial_u\phi(u,z)+\Delta^{\alpha/2}_z \phi(u,z)\right]\,dzdu
 = -\phi(s,x)\,,
\end{equation*}
where $s\in \RR$, $x\in  \Rd$, and $\phi\in C^\infty_c(\RR\times
\Rd)$ (\cite{2008KBTJ}).
We denote
$$\nu(y)=\frac{2^{\alpha}\Gamma((d+\alpha)/2)}{\pi^{d/2}|\Gamma(-\alpha/2)|}|y|^{-d-\alpha}\,,$$
the density function of the L\'evy measure of the
semigroup $\{P_t\}$ \cite{MR1739520, MR2013738, KB-TB-MR-TK-RS-ZV}.

There is a constant $c$
such that (see \cite{MR2013738} or \cite{MR119247})
\begin{equation}
  \label{eq:setd}
c^{-1} \left(\frac{t}{|x|^{d+\alpha}} \land
        t^{-d/\alpha}\right) \le p_t(x) \le c
        \left(\frac{t}{|x|^{d+\alpha}} \land t^{-d/\alpha}\right)\,,
\quad x\in  \Rd\,,\;t>0\,.
\end{equation}
(\ref{eq:setd}) and similar {\it sharp estimates} 
(i.e. such that the lower and upper bounds are comparable) will be
abbreviated as follows:
\begin{equation}
  \label{eq:ppfe}
p_t(x)\approx t^{-d/\alpha}\wedge \frac{t}{|x|^{d+\alpha}}
\,, \quad x\in \Rd\,,t>0\,.
\end{equation}
The {standard} isotropic $\alpha$-stable L\'evy
process $(X_t,P_x)$ on $\Rd$ may be constructed by specifying 
the following time-homogeneous transition probability:
$$
P_t(x,A)=\int_A p_{t}(y-x)dy\,,\quad t>0\,,\;x\in \Rd\,,\; A\subset \Rd\,,
$$
and stipulating that $P_x(X(0)=x)=1$. Thus, $P_x$, $E_x$ denote
the distribution and expectation for the process starting from
$x$. The distribution of the process is concentrated on right
continuous functions: $[0,\infty)\to \Rd$ with left limits, and for all $s\geq 0$, $x\in
\Rd$ we have that $P_x(X_s=X_{s-})=1$. It is well-known that
$(X_{t},P_{x})$ is strong Markov with respect to the so-called
standard filtration \cite{MR1406564, MR0264757}. The L\'evy system
(see \cite[VII.68]{MR745449}, \cite[Appendix A]{MR2357678}, see also
\cite[Theorem 2.4]{MR2255353}, \cite[Corollary 2.8]{MR2231884} and
\cite[Lemma 1]{MR2345912}) for $(X_{t},P^{x})$ amounts to the
equality,
\begin{equation}
  \label{eq:Ls}
E_x\left[\sum_{s\leq T}f(s,X_{s-}, X_s)\right]
=E_x\left[\int_0^T\left(\int_\Rd
f(s,X_{s},y)\nu(w-X_{s})dw\right)ds\right]\,,
\end{equation}
where $x\in \Rd$, $f\geq 0$ is a Borel function on
$\RR\times \Rd\times \Rd$, such that $f(s,z,w)=0$ if $z=w$, and $T$
is a stopping time with respect to the filtration of $X$.

For open $D\subset \Rd$ we let
$
\tau_D=\inf\{t>0: \, X_t\notin D\}
$, and we define
$$
p^D_t(x,y)=p_t(x,y)-E_x[\tau_D<t;\,
p_{t-\tau_D}(X_{\tau_D},y)],\quad x,y\in \Rd,\, t>0\,,
$$
see, e.g., \cite{CKS2008, MR1671973}. 
Clearly,
\begin{equation}\label{eq:**}
0\leq p^D_t(x,y)\leq p_t(y-x)\,.
\end{equation}
By the strong Markov property, $p^D_t$ is the
transition density of the isotropic stable process {\it killed} on leaving $D$, meaning that
$p^D$ satisfies the Chapman-Kolmogorov equation:
$$
\int_\Rd p_s^D(x,z)p_t^D(z,y)dz=p_{s+t}^D(x,y)\,,\quad x,y\in \Rd,\, s,t>0\,,
$$
and for every $x\in\Rd$, $t>0$ and bounded Borel function $f$,
$$\int_\Rd f(y)p^D_t(x,y)dy=E_x[\tau_D<t;\, f(X_t)]\,.$$
Furthermore, for $s\in \RR$, $x\in  \Rd$, and $\phi\in C^\infty_c(\RR\times
D)$, we have
\begin{equation*}
\int\limits_{s}^\infty\int\limits_{D}
p^D_{u-s}(x,z)\left[
\partial_u\phi(u,z)+\Delta^{\alpha/2}_z \phi(u,z)\right]\,dzdu
 = -\phi(s,x)\,,
\end{equation*}
which justifies calling $p^D$ the heat kernel of the fractional
Laplacian {\it on} $D$.
In analogy with (\ref{eq:sca}) we have the following scaling property
\begin{equation}
  \label{eq:scD}
  p^D_t(x,y)=t^{-d/\alpha}
  p_1^{t^{-1/\alpha}D}(t^{-1/\alpha}x,t^{-1/\alpha}y)\,,\quad x,y\in \Rd\,,\; t>0\,.
\end{equation}
\section{$C^{1,1}$ domains}\label{s:s}

Let $D\subset \Rd$ be a $C^{1,1}$ domain, meaning that $D$ is open and there is
$r_0>0$ such that for every $z\in \partial D$ there exist balls
$B_z(r_0)\subset D$ and $B_z'(r_0)\subset D^c$ of radius $r_0$,
{\it tangent} at $z$. Denote $\delta_D(x)={\rm dist}(x,D^c)$, the
distance to $D^c$. 
The transition density of the stable L\'evy process killed off $D$
satisfies (\cite{CKS2008})
  \begin{equation}
    \label{eq:CKS}
    p_t^D(x,y)\approx
    \left(1\wedge\frac{\delta_D^{\alpha/2}(x)}{\sqrt{t}}\right)
\left(1\wedge\frac{\delta_D^{\alpha/2}(y)}{\sqrt{t}}\right)
p_t(x,y)\,,\quad 0<t\leq
    1\,,\;x,y\in \Rd\,.
  \end{equation}
\begin{cor} If $D$ is a $C^{1,1}$ domain then
    \begin{equation}
    \label{eq:CKS0}
    P_x(\tau_D>t)\approx    1\wedge\frac{\delta_D^{\alpha/2}(x)}{\sqrt{t}}\,,\quad 0<t\leq
    1\,,\;x,y\in \Rd\,.
  \end{equation}
\end{cor}
\begin{proof}
We have
$$
P_x(\tau_D>t)=\int_\Rd p_t^D(x,y)dy\,.
$$
By (\ref{eq:CKS}),
\begin{eqnarray*}
P_x(\tau_D>t)&=& \int_D p_t^D(x,y)dy \approx
\left(1\wedge\frac{\delta_D^{\alpha/2}(x)}{\sqrt{t}}\right) I_t(x)
\,,\quad 0<t\leq
    1\,,\;x,y\in \Rd\,,
\end{eqnarray*}
where
$$I_t(x)=\int_D
\left(1\wedge\frac{\delta_D^{\alpha/2}(y)}{\sqrt{t}}\right)p_t(x,y)dy\,.$$
Clearly, $I_t(x)\leq\int_{\Rd} p_t(x,y)dy=1$. This yields the
upper bound in (\ref{eq:CKS0}). To prove the lower bound we consider
$0<t\leq 1$ and we will first
assume that $\delta_D(x)>t^{1/\alpha}$. If
$|y-x|<t^{1/\alpha}/2$, then $p_t(x,y)\approx t^{-d/\alpha}$, and
we get
\begin{eqnarray*}
I_t(x)\geq c\int_{|y-x|<t^{1/\alpha}/2}
t^{-d/\alpha}\,
dy=c>0\,.
\end{eqnarray*}
If $\delta_D(x)\leq t^{1/\alpha}$, then let $z\in \partial D$ be
such that $|x-z|=\delta_D(x)$, and consider the inner tangent ball
$B_z(t^{1/\alpha}\wedge r_0)$ for $D$ at $z$, with center at, say,
$w$. We have
$$
I_t(x)\geq \int_{B_z(t^{1/\alpha}\wedge r_0)}
 \frac{(t^{1/\alpha}-|y-w|)^{\alpha/2}}{\sqrt{t}} p_t(x,y)\,dy\,.
$$
Since
$p_t(x,y)=t^{-d/\alpha}p_1(\frac{x-w}{t^{1/\alpha}}
,\frac{y-w}{t^{1/\alpha}})$,
by changing variable $v=\frac{(y-w)}{t^{1/\alpha}}$, we get
$$
I_t(x)\geq \int_{B(0,1\wedge r_0)} (1-|v|)^{\alpha/2}
p_1(u,v)\,dv\,,
$$
where $u=t^{-1/\alpha}(x-w)\in B(0,1)$.
The latter integral is continuous and strictly positive for $u\in
\overline{B(0,1)}$. Thus, $\inf_{x\in D}I_t(x)>0$.
The proof of (\ref{eq:CKS0}) is complete.
\end{proof}

\begin{cor} \label{cor:c11}
If $D$ is a $C^{1,1}$ domain then
  \begin{equation*}
    p_t^D(x,y)\approx P_x(\tau_D>t)P_y(\tau_D>t)p_t(x,y)\,,\quad 0<t\leq
    1\,,\;x,y\in \Rd\,.
  \end{equation*}
\end{cor}
\section{Cones}\label{s:c}
For $x\in \Rd\setminus\{0\}$ we denote by $\theta(x)$ the angle
between $x$ and the point $(0,\ldots,0,1)\in \Rd$.
We fix $0<\Theta<\pi$ and consider
the right circular cone
$\Gamma=\{x\in \Rd\setminus \{0\}:\, \theta(x)<\Theta\}$.
Clearly, $r\Gamma=\Gamma$ for every $r>0$.
By (\ref{eq:scD}),
\begin{equation}
  \label{eq:scG}
  p^\Gamma_t(x,y)=t^{-d/\alpha}
  p_1^{\Gamma}(t^{-1/\alpha}x,t^{-1/\alpha}y)\,,\quad x,y\in \Rd\,,\; t>0\,.
\end{equation}
We fix $x_0\in \Gamma$ and consider the Martin kernel $M$ for $\Gamma$ with the pole at
infinity, so normalized that $M(x_0)=1$.
It is known that there is $0\leq\beta<\alpha$ such that
\begin{equation*}
M(x)=|x|^\beta M(x/|x|)\,,\quad x\neq 0\,,
\end{equation*}
see \cite{MR2075671}, \cite{MR2213639}, \cite{MR2182071}.
Since the boundary of $\Gamma$ is smooth except at the origin, by \cite[Lemma 3.3]{MR2213639},
\begin{equation}
  \label{eq:ojmcc}
M(x)\approx \delta_\Gamma(x)^{\alpha/2}|x|^{\beta-\alpha/2}\,,\quad
  x\in \Rd\,.
\end{equation}
The following result strengthens \cite[Lemma 4.2]{MR2075671}.
\begin{lem}\label{lem:eetG}
If $\Gamma$ is a right circular cone then
\begin{equation}
    \label{eq:eetGs}
P_x(\tau_\Gamma>t) \approx \left(
\delta_\Gamma^{\alpha/2}(t^{-1/\alpha}x)\wedge
1\right)\left(|t^{-1/\alpha}x|\wedge 1\right)^{\beta-\alpha/2}
\,,\quad x\in \Rd\,,\;t>0\,.
  \end{equation}
\end{lem}
\begin{proof}
Since $P_x(\tau_\Gamma>t)=P_{t^{-1/\alpha}x}(\tau_\Gamma>1)$,
we only need to prove that
  \begin{equation}
    \label{eq:eetG}
P_x(\tau_\Gamma>1) \approx \left(\delta_\Gamma^{\alpha/2}(x)\wedge
1 \right) \left(|x|\wedge
  1\right)^{\beta-\alpha/2}\,,
\quad x\in \Rd\,.
\end{equation}
If $|x|<1$ then (\ref{eq:eetG})
is a consequence of (\ref{eq:ojmcc}) and \cite[Lemma
4.2]{MR2075671}. If $|x|\geq 1$ then $P_x(\tau_\Gamma>1)\approx
\delta_\Gamma(x)^{\alpha/2}\wedge 1$. Indeed, considering
$C^{1,1}$ domains $\Gamma'$ and $\Gamma''$ such that
$\Gamma'\subset \Gamma \subset \Gamma''$ and $\Gamma''\setminus
\Gamma'\subset B(0,1/2)$, we see that $\delta_{\Gamma'}(x)\leq
\delta_{\Gamma}(x)\leq \delta_{\Gamma''}(x)\leq
2\delta_{\Gamma'}(x)$ for such $x$. Since
$P_x(\tau_{\Gamma'}>1)\leq P_x(\tau_{\Gamma}>1)\leq
P_x(\tau_{\Gamma''}>1)$, by using (\ref{eq:CKS0}) we obtain
(\ref{eq:eetG}).
\end{proof}
An interesting, if trivial, consequence of (\ref{eq:eetGs}) is that
\begin{equation}
  \label{eq:iit}
P_x(\tau_\Gamma>t) \approx P_x(\tau_\Gamma>t/2)\,,\quad t>0\,,\; x\in \Rd\,.
\end{equation}

\begin{thm}\label{prop:ppu}
  \begin{equation}
    \label{eq:ppu}
p^\Gamma_t(x,y)\approx P_x(\tau_\Gamma>t) P_y(\tau_\Gamma>t)
p_t(x,y)\,,\quad x,y\in\Rd\,,\;t>0\,.
  \end{equation}
\end{thm}
\begin{proof}
We note that the right hand side, say $R_t(x,y)$, of (\ref{eq:ppu}) satisfies
$$
R_t(x,y)=t^{-d/\alpha}R_1(t^{-1/\alpha}x,t^{-1/\alpha}y)\,.
$$
Thus, in view of (\ref{eq:scG}), we only need to prove
(\ref{eq:ppu}) for $t=1$.

Let $\Gamma'$ and $\Gamma''$ be as in the proof of Lemma
\ref{lem:eetG}. Then
\begin{equation*}p^{\Gamma'}_1(x,y)\leq p^{\Gamma}_1(x,y)\leq
p^{\Gamma''}_1(x,y).
\end{equation*}
By (\ref{eq:CKS}) we have, for $|x|,|y|\geq 1$,
$$p^{\Gamma'}_1(x,y)\approx \left(1\wedge\delta_{\Gamma}^{\alpha/2}(x)\right)
\left(1\wedge\delta_{\Gamma}^{\alpha/2}(y)\right) p_1(x,y)
\approx p^{\Gamma''}_1(x,y)\,.$$
Hence by Lemma \ref{lem:eetG} we obtain,
\begin{equation}\label{eq:ppu1} p^{\Gamma}_1(x,y)\approx P_x(\tau_\Gamma>1)
P_y(\tau_\Gamma>1) p_1(x,y)\,,\quad |x|,|y|\geq 1.
\end{equation}
In particular, there is a constant $c$ such that
\begin{equation}\label{eq:ppu2}
p^{\Gamma}_1(x,y)\leq c P_x(\tau_\Gamma>1) p_1(x,y)\,,\quad |x|,|y|\geq 1\,.
\end{equation}
If $|x|< 1$ and $|y|\leq 4$,
then $|x-y|<5$ and
$p_1(x,y)\geq c(1\wedge |x-y|^{-d-\alpha})\geq c$. By
the semigroup property (\ref{eq:**}) and (\ref{eq:iit}),
\begin{eqnarray}\label{eq:ppu3}
p^{\Gamma}_1(x,y)
&=&\int_{\Gamma}p^{\Gamma}_{1/2}(x,w)p^{\Gamma}_{1/2}(w,y)dw\leq c \int_{\Gamma}p^{\Gamma}_{1/2}(x,w)dw\notag\\ &=& c
P_x(\tau_{\Gamma}>1/2)\leq cP_x(\tau_{\Gamma}>1)\notag\\&\leq& c
P_x(\tau_{\Gamma}>1)p_1(x,y)\,, \quad |x|< 1\,,\;|y|\leq 4\,.
\end{eqnarray}
We next assume that $|x|< 1$ and $|y|>4$. Then $|x-y|>3$ and
$p_1(x,y)\approx |x-y|^{-d-\alpha}$. Denote $\Gamma_1=\Gamma\cap
B(0,2)$, $\Gamma_2=(\Gamma\setminus\Gamma_1)\cap B(0,(|y|+1)/2) $
and $\Gamma_3=\Gamma\setminus B(0,(|y|+1)/2)$. Using the strong
Markov property and the L\'evy system (\ref{eq:Ls}) with
$f(s,z,w)=
\textbf{1}_{\Gamma_1}(z)
\textbf{1}_{\Gamma_1^c}(w)
p^{\Gamma}_{1-s}(w,y)$ 
and
$T=1\wedge \tau_{\Gamma_1}$, 
we obtain
\begin{eqnarray*}
p^{\Gamma}_1(x,y)
&=&E_x[\tau_{\Gamma_1}<1;p^{\Gamma}_{1-\tau_{\Gamma_1}}(X_{\tau_{\Gamma_1}},y)]\nonumber\\
&=&\int^1_0\int_{\Gamma_1}p^{\Gamma_1}_s(x,z)\int_{\Gamma\setminus\Gamma_1}
\nu(w-z)p^{\Gamma}_{1-s}(w,y)\,dwdz
ds\\
&=&\int^1_0\int_{\Gamma_1}p^{\Gamma_1}_s(x,z)\int_{\Gamma_2}
\nu(w-z)p^{\Gamma}_{1-s}(w,y)\,dwdzds\nonumber\\
& &+
\int^1_0\int_{\Gamma_1}p^{\Gamma_1}_s(x,z)\int_{\Gamma_3}
\nu(w-z)p^{\Gamma}_{1-s}(w,y)\,dwdzds\nonumber\\
&=&I+II\,.\nonumber
\end{eqnarray*}
We note that for $w\in \Gamma_2$,
$$|w-y|\geq |y|-|w|\geq |y|/4 \geq |x-y|/8\,.$$
Since $p^{\Gamma_1}_{1-s}(w-y)\leq p_{1-s}(w-y)\leq
c(1-s)|w-y|^{-d-\alpha}$, we obtain
\begin{eqnarray*}
I&\leq&\int^1_0\int_{\Gamma_1}p^{\Gamma_1}_s(x,z)
\int_{\Gamma_2}\nu(w-z)c \frac{1-s}{|w-y|^{d+\alpha}}\,dwdzds\\
&\leq &
c|x-y|^{-d-\alpha}\int^1_0\int_{\Gamma_1}p^{\Gamma_1}_s(x,z)
\int_{\Gamma\setminus\Gamma_1}\nu(w-z)\,dwdzds\\
&=& c|x-y|^{-d-\alpha}P_x(X_{\tau_{\Gamma_1}}\in
\Gamma\setminus\Gamma_1,\tau_{\Gamma_1}\leq 1)\\
&\leq&
c|x-y|^{-d-\alpha}P_x(X_{\tau_{\Gamma_1}}\in
\Gamma\setminus\Gamma_1)\approx M(x)p_1(x,y)\,.
\end{eqnarray*}
In the last line we used BHP (\cite{MR2075671}).
For $z\in \Gamma_1$ and $w\in \Gamma_3$, we have
$$|w-z|\geq |w|-|z|\geq |y|/2-3/2\geq |y|/8\geq |x-y|/16,$$
hence $\nu(w-z)\leq c |x-y|^{-d-\alpha}$, and so
\begin{eqnarray*}
II&\leq &c|x-y|^{-d-\alpha}
\int^1_0\int_{\Gamma_1}p^{\Gamma_1}_s(x,z)
\int_{\Gamma_3}p^{\Gamma}_{1-s}(w,y)\,dwdzds\\
&\leq&c|x-y|^{-d-\alpha}
\int^1_0\int_{\Gamma_1}p^{\Gamma_1}_s(x,z)dzds\leq
c|x-y|^{-d-\alpha}E_x\tau_{\Gamma_1}\\&\leq& cp_1(x,y)M(x),
\end{eqnarray*}
where the last inequality follows from (\ref{eq:ppfe}) and \cite[Lemma 4.6]{MR2075671}.
By Lemma \ref{lem:eetG},
\begin{equation}\label{eq:ppu4} p^{\Gamma}_1(x,y)\leq c
P_x(\tau_\Gamma>1) p_1(x,y), \quad |x|<1,|y|>4.
\end{equation}
Combining (\ref{eq:ppu2}), (\ref{eq:ppu3}) and (\ref{eq:ppu4}),  we get
\begin{equation*}
p^{\Gamma}_1(x,y)
\leq c P_x(\tau_\Gamma>1) p_1(x,y), \quad x,y\in
\Rd.
\end{equation*}
By the symmetry, semigroup property and (\ref{eq:iit}) we obtain
\begin{eqnarray*}p^{\Gamma}_1(x,y)&=&\int_{\Gamma}p^{\Gamma}_{1/2}(x,w)p^{\Gamma}_{1/2}(w,y)dw\\
&=&\int_{\Gamma}2^{2d/\alpha}p^{\Gamma}_{1}(x2^{1/\alpha},w2^{1/\alpha})p^{\Gamma}_{1}(w2^{1/\alpha},y2^{1/\alpha})dw\\
&\leq& c P_{x2^{1/\alpha}}(\tau_{\Gamma}>1)P_{y2^{1/\alpha}}(\tau_{\Gamma}>1)\int_{\Rd}p_{1/2}(x,w)p_{1/2}(w,y)dw\\
&\leq& c P_{x}(\tau_{\Gamma}>1)P_{y}(\tau_{\Gamma}>1)p_{1}(x,y).
\end{eqnarray*}

We will now prove the lower bound in (\ref{eq:ppu}).
We first assume that $|x|< 1$ and $|y|\leq 2$, and we
let $\Gamma_4=\Gamma\cap B(0,4)$. Since $\Gamma_4$ is bounded,
the semigroup $p^{\Gamma_4}_t$ is intrinsically ultracontractive (\cite{MR1643611}).
In particular,
\begin{equation*}
p^{\Gamma}_{1/2}(x,y)\geq p^{\Gamma_4}_{1/2}(x,y)\geq c
E_x\tau_{\Gamma_4}E^y\tau_{\Gamma_4}\,.
\end{equation*}
Furthermore, by \cite[Lemma 4.6]{MR2075671} and Lemma \ref{lem:eetG} we obtain
$$E^y\tau_{\Gamma_4}\geq c M(y)\geq c P_y(\tau_{\Gamma}>1)\,.$$
We see that
\begin{equation}\label{eq:ppd1}
p^{\Gamma}_{1/2}(x,y)\geq c
P_x(\tau_{\Gamma}>1)P_y(\tau_{\Gamma}>1)p_1(x,y),\quad
|x|<1,\;|y|\leq 2\,.
\end{equation}
If $|x|<1$ and $|y|>2$, then by the semigroup property,
(\ref{eq:ppu1}) and (\ref{eq:ppd1}),
\begin{eqnarray*}p^{\Gamma}_1(x,y)&=&\int_{\Gamma}p^{\Gamma}_{1/2}(x,z)p^{\Gamma}_{1/2}(z,y)dz\\
&\geq&
c\,P_x(\tau_{\Gamma}>1)P_y(\tau_{\Gamma}>1)
\int_{\Gamma_1\setminus B(0,1)}P_z(\tau_{\Gamma}>1)^2p_1(x,z)p_1(z,y)dz\\
&\geq& c\,
P_x(\tau_{\Gamma}>1)P_y(\tau_{\Gamma}>1)p_1(x,y)\int_{\Gamma_1\setminus
  B(0,1)}P_z(\tau_{\Gamma}>1)^2dz\,.
\end{eqnarray*}
Hence
\begin{equation}\label{eq:ppd2}
p^{\Gamma}_1(x,y)\geq c
P_x(\tau_{\Gamma}>1)P_y(\tau_{\Gamma}>1)p_1(x,y),\quad
|x|<1,|y|>2.
\end{equation}
By (\ref{eq:ppu1}), (\ref{eq:ppd1}), (\ref{eq:ppd2}), symmetry (and scaling), we get
the lower bound in (\ref{eq:ppu}).
\end{proof}
We like to note that Theorem~\ref{prop:ppu} strengthens \cite[Corollary 4.8]{MR2075671}.
Also,
\begin{equation*}
p_t^\Gamma(x,y)\approx p_{t/2}^\Gamma(x,y)\,,\quad x,y\in \Rd\,,\; t>0\,.
\end{equation*}

In view of  Lemma~\ref{lem:eetG}, for the right circular cone
$\Gamma$, (\ref{eq:ppu}) is equivalent to
\begin{eqnarray}
\!\!p^\Gamma_t(x,y)\!\!\!\!\!\!
&\approx&\!\!\!\!\!\!
\left(\delta_\Gamma^{\alpha/2}(t^{-1/\alpha}x)\wedge
1\right)\left(|t^{-1/\alpha}x|\wedge 1\right)^{\beta-\alpha/2}\; 
\left( t^{-d/\alpha}\wedge \frac{t}{|x-y|^{d+\alpha}}\right)
\label{eq:etdG2}\\
&&\left( \delta_\Gamma^{\alpha/2}(t^{-1/\alpha}y)\wedge
1\right)\left(|t^{-1/\alpha}y|\wedge 1\right)^{\beta-\alpha/2}\,,\quad
t>0\,,\;x,y\in \Rd\,. \nonumber
\end{eqnarray}
This is explicit except for the exponent $\beta$ (see \cite{MR2075671} in this connection).
Recall that $\int_0^\infty p_t^\Gamma(x,y)dt=G_\Gamma(x,y)$, the Green function of $\Gamma$.
By integrating (\ref{eq:etdG2}) one can obtain sharp estimates for
the Green function of the right circular cone.
For $d\geq 2$ the estimates--first given in \cite[Theorem 3.10]{MR2213639}--are the following, 
\begin{equation}
  \label{eq:eGf}
\frac{G_\Gamma(x,y)}
{|x-y|^{\alpha-d}}
\approx
1\wedge
\left\{
\frac{\delta^{\alpha/2}_\Gamma(x)\delta^{\alpha/2}_\Gamma(y)}{|y-x|^\alpha}
\left(
\frac{|x|\wedge |y|}{|x|\vee |y|}
\right)^{\beta-\alpha/2}
\right\}
\,,\quad x,y\in \Rd\,.
\end{equation}
We skip the details of the integration (similar calculations are given
in \cite{MR2255353} and \cite{CKS2008}).
Noteworthy, $\beta=\alpha/2$ if $\Gamma$ is a half-space \cite{MR2075671}.
For the case of dimension
$d=1$, and $\Gamma=(0,\infty)$, we refer the reader to
\cite{MR1911445}, see also \cite[Corollary 1.2]{CKS2008}. 

As stated in Introduction, we expect (\ref{eq:*g}) to be true quite
generally. In particular the approximation should hold for domains above the graph of a Lipschitz
function for all times $t>0$. Corollary~\ref{cor:c11} confirms this conjecture for
$C^{1,1}$ domains and small times, while Theorem~\ref{prop:ppu} proves
it for the right circular cones and all times. 
By inspecting the relevant proofs in \cite{MR2213639}, the reader may also verify
without difficulty that Theorem~\ref{prop:ppu}
and (\ref{eq:eGf}) hold the same for all those {\it generalized} cones (\cite{MR2075671}) 
in $\Rd$, $d\geq 2$, which are $C^{1,1}$ except at the origin. 

On the other hand, if $D$ is a bounded $C^{1,1}$ domain, and if we
denote by $-\lambda_1$ the first eigenvalue of $\Delta^{\alpha/2}$ on
$D$ (i.e. when acting on functions vanishing off $D$),
then by the intrinsic ultracontractivity (see, e.g., \cite[Theorem 1.1]{CKS2008}),
\begin{equation*}
    p_t^D(x,y)\approx \delta_D^{\alpha/2}(x)\delta^{\alpha/2}(y) e^{-\lambda_1 t}\,,\quad t>1\,,\;x,y\in \Rd\,,
\end{equation*}
and so
  \begin{equation*}
    P_x(\tau_D>t)\approx \delta_D^{\alpha/2}(x) e^{-\lambda_1 t}\,,\quad t>1\,,\;x\in \Rd\,.
  \end{equation*}
Therefore (\ref{eq:*g}) fails for large times $t$ if $D$ is bounded
(see also \cite{MR1643611}). 


\noindent
{\bf Acknowledgments.} We thank Micha{\l} Ryznar and Mateusz
Kwa\'snicki for discussions and suggestions.

\bibliographystyle{abbrv}

\begin{thebibliography}{10}

\bibitem{MR2153910}
H.~Aikawa and T.~Lundh.
\newblock The 3{G} inequality for a uniformly {J}ohn domain.
\newblock {\em Kodai Math. J.}, 28(2):209--219, 2005.

\bibitem{MR2075671}
R.~Ba{\~n}uelos and K.~Bogdan.
\newblock Symmetric stable processes in cones.
\newblock {\em Potential Anal.}, 21(3):263--288, 2004.

\bibitem{MR2345912}
R.~Ba{\~n}uelos and K.~Bogdan.
\newblock L\'evy processes and {F}ourier multipliers.
\newblock {\em J. Funct. Anal.}, 250(1):197--213, 2007.

\bibitem{MR2438694}
R.~Ba{\~n}uelos and T.~Kulczycki.
\newblock Trace estimates for stable processes.
\newblock {\em Probab. Theory Related Fields}, 142(3-4):313--338, 2008.

\bibitem{MR1465162}
R.~Ba{\~n}uelos and R.~G. Smits.
\newblock Brownian motion in cones.
\newblock {\em Probab. Theory Related Fields}, 108(3):299--319, 1997.

\bibitem{2008-MB-RB-TK}
M.~Barlow, A.~Grigor'yan, and T.~Kumagai.
\newblock Heat kernel upper bounds for jump processes and the first exit time.
\newblock to appear Math. Z.

\bibitem{MR0481057}
C.~Berg and G.~Forst.
\newblock {\em Potential theory on locally compact abelian groups}.
\newblock Springer-Verlag, New York, 1975.
\newblock Ergebnisse der Mathematik und ihrer Grenzgebiete, Band 87.

\bibitem{MR1406564}
J.~Bertoin.
\newblock {\em L\'evy processes}, volume 121 of {\em Cambridge Tracts in
  Mathematics}.
\newblock Cambridge University Press, Cambridge, 1996.

\bibitem{MR119247}
R.~M. Blumenthal and R.~K. Getoor.
\newblock Some theorems on stable processes.
\newblock {\em Trans. Amer. Math. Soc.}, 95:263--273, 1960.

\bibitem{MR0264757}
R.~M. Blumenthal and R.~K. Getoor.
\newblock {\em Markov processes and potential theory}.
\newblock Pure and Applied Mathematics, Vol. 29. Academic Press, New York,
  1968.

\bibitem{MR1741527}
K.~Bogdan.
\newblock Sharp estimates for the {G}reen function in {L}ipschitz domains.
\newblock {\em J. Math. Anal. Appl.}, 243(2):326--337, 2000.

\bibitem{MR2006232}
K.~Bogdan, K.~Burdzy, and Z.-Q. Chen.
\newblock Censored stable processes.
\newblock {\em Probab. Theory Related Fields}, 127(1):89--152, 2003.

\bibitem{MR1671973}
K.~Bogdan and T.~Byczkowski.
\newblock Potential theory for the {$\alpha$}-stable {S}chr\"odinger operator
  on bounded {L}ipschitz domains.
\newblock {\em Studia Math.}, 133(1):53--92, 1999.

\bibitem{MR1825645}
K.~Bogdan and T.~Byczkowski.
\newblock Potential theory of {S}chr\"odinger operator based on fractional
  {L}aplacian.
\newblock {\em Probab. Math. Statist.}, 20(2, Acta Univ. Wratislav. No.
  2256):293--335, 2000.

\bibitem{KB-TB-MR-TK-RS-ZV}
K.~Bogdan, T.~Byczkowski, K.~Tadeusz, M.~Ryznar, R.~Song, and
  Z.~Vondra{\v{c}}ek.
\newblock Potential analysis of stable processes and its extensions.
\newblock Based on lectures given on the CNRS/HARP Workshop Stochastic and
  Harmonic Analysis of Processes with Jumps Angers, May 2-9, 2006.

\bibitem{2008KBTJ}
K.~Bogdan, W.~Hansen, and T.~Jakubowski.
\newblock Time-dependent {S}chr{\"o}dinger perturbations of transition
  densities.
\newblock {\em Studia Mathematica}, 189(3):235--254, 2008.

\bibitem{MR2182071}
K.~Bogdan and T.~Jakubowski.
\newblock Probl\`eme de {D}irichlet pour les fonctions {$\alpha$}-harmoniques
  sur les domaines coniques.
\newblock {\em Ann. Math. Blaise Pascal}, 12(2):297--308, 2005.

\bibitem{MR2283957}
K.~Bogdan and T.~Jakubowski.
\newblock Estimates of heat kernel of fractional {L}aplacian perturbed by
  gradient operators.
\newblock {\em Comm. Math. Phys.}, 271(1):179--198, 2007.

\bibitem{MR2365478}
K.~Bogdan, T.~Kulczycki, and M.~Kwa{\'s}nicki.
\newblock Estimates and structure of {$\alpha$}-harmonic functions.
\newblock {\em Probab. Theory Related Fields}, 140(3-4):345--381, 2008.

\bibitem{MR2013738}
K.~Bogdan, A.~St{\'o}s, and P.~Sztonyk.
\newblock Harnack inequality for stable processes on {$d$}-sets.
\newblock {\em Studia Math.}, 158(2):163--198, 2003.

\bibitem{MR2320691}
K.~Bogdan and P.~Sztonyk.
\newblock Estimates of the potential kernel and {H}arnack's inequality for the
  anisotropic fractional {L}aplacian.
\newblock {\em Studia Math.}, 181(2):101--123, 2007.

\bibitem{MR799436}
K.~Burdzy.
\newblock Brownian paths and cones.
\newblock {\em Ann. Probab.}, 13(3):1006--1010, 1985.

\bibitem{MR0474525}
D.~L. Burkholder.
\newblock Exit times of {B}rownian motion, harmonic majorization, and {H}ardy
  spaces.
\newblock {\em Advances in Math.}, 26(2):182--205, 1977.

\bibitem{MR1911445}
H.~Byczkowska and T.~Byczkowski.
\newblock One-dimensional symmetric stable {F}eynman-{K}ac semigroups.
\newblock {\em Probab. Math. Statist.}, 21(2, Acta Univ. Wratislav. No.
  2328):381--404, 2001.

\bibitem{MR1942325}
Z.-Q. Chen and P.~Kim.
\newblock Green function estimate for censored stable processes.
\newblock {\em Probab. Theory Related Fields}, 124(4):595--610, 2002.

\bibitem{CKS2008}
Z.-Q. Chen, P.~Kim, and R.~Song.
\newblock Heat kernel estimates for {D}irichlet fractional {L}aplacian.
\newblock to appear in J. European Math. Soc.

\bibitem{CKS2008censored}
Z.-Q. Chen, P.~Kim, and R.~Song.
\newblock Two-sided heat kernel estimates for censored stable-like processes.
\newblock to appear in Probab. Theory Related Fields, 2008.

\bibitem{MR2357678}
Z.-Q. Chen and T.~Kumagai.
\newblock Heat kernel estimates for jump processes of mixed types on metric
  measure spaces.
\newblock {\em Probab. Theory Related Fields}, 140(1-2):277--317, 2008.

\bibitem{MR1654824}
Z.-Q. Chen and R.~Song.
\newblock Estimates on {G}reen functions and {P}oisson kernels for symmetric
  stable processes.
\newblock {\em Math. Ann.}, 312(3):465--501, 1998.

\bibitem{MR2255354}
S.~Cho.
\newblock Two-sided global estimates of the {G}reen's function of parabolic
  equations.
\newblock {\em Potential Anal.}, 25(4):387--398, 2006.

\bibitem{MR902422}
K.~L. Chung.
\newblock Green's function for a ball.
\newblock In {\em Seminar on stochastic processes, 1986 ({C}harlottesville,
  {V}a., 1986)}, volume~13 of {\em Progr. Probab. Statist.}, pages 1--13.
  Birkh\"auser Boston, Boston, MA, 1987.

\bibitem{MR879702}
E.~B. Davies.
\newblock The equivalence of certain heat kernel and {G}reen function bounds.
\newblock {\em J. Funct. Anal.}, 71(1):88--103, 1987.

\bibitem{MR863716}
R.~D. DeBlassie.
\newblock Exit times from cones in {${\bf R}\sp n$} of {B}rownian motion.
\newblock {\em Probab. Theory Related Fields}, 74(1):1--29, 1987.

\bibitem{MR1062058}
R.~D. DeBlassie.
\newblock The first exit time of a two-dimensional symmetric stable process
  from a wedge.
\newblock {\em Ann. Probab.}, 18(3):1034--1070, 1990.

\bibitem{MR745449}
C.~Dellacherie and P.-A. Meyer.
\newblock {\em Probabilities and potential. {B}}, volume~72 of {\em
  North-Holland Mathematics Studies}.
\newblock North-Holland Publishing Co., Amsterdam, 1982.
\newblock Theory of martingales, Translated from the French by J. P. Wilson.

\bibitem{MR2085428}
B.~Dyda.
\newblock A fractional order {H}ardy inequality.
\newblock {\em Illinois J. Math.}, 48(2):575--588, 2004.

\bibitem{MR2430977}
A.~Grigor{\cprime}yan and J.~Hu.
\newblock Off-diagonal upper estimates for the heat kernel of the {D}irichlet
  forms on metric spaces.
\newblock {\em Invent. Math.}, 174(1):81--126, 2008.

\bibitem{MR657523}
M.~Gr{\"u}ter and K.-O. Widman.
\newblock The {G}reen function for uniformly elliptic equations.
\newblock {\em Manuscripta Math.}, 37(3):303--342, 1982.

\bibitem{MR2417435}
T.~Grzywny and M.~Ryznar.
\newblock Estimates of {G}reen functions for some perturbations of fractional
  {L}aplacian.
\newblock {\em Illinois J. Math.}, 51(4):1409--1438, 2007.

\bibitem{MR2386098}
T.~Grzywny and M.~Ryznar.
\newblock Two-sided optimal bounds for {G}reen functions of half-spaces for
  relativistic {$\alpha$}-stable process.
\newblock {\em Potential Anal.}, 28(3):201--239, 2008.

\bibitem{MR2238879}
Q.-Y. Guan.
\newblock Integration by parts formula for regional fractional {L}aplacian.
\newblock {\em Comm. Math. Phys.}, 266(2):289--329, 2006.

\bibitem{MR2160104}
W.~Hansen.
\newblock Uniform boundary {H}arnack principle and generalized triangle
  property.
\newblock {\em J. Funct. Anal.}, 226(2):452--484, 2005.

\bibitem{MR2207878}
W.~Hansen.
\newblock Global comparison of perturbed {G}reen functions.
\newblock {\em Math. Ann.}, 334(3):643--678, 2006.

\bibitem{MR1873235}
N.~Jacob.
\newblock {\em Pseudo differential operators and {M}arkov processes. {V}ol.
  {I}}.
\newblock Imperial College Press, London, 2001.
\newblock Fourier analysis and semigroups.

\bibitem{MR1991120}
T.~Jakubowski.
\newblock The estimates for the {G}reen function in {L}ipschitz domains for the
  symmetric stable processes.
\newblock {\em Probab. Math. Statist.}, 22(2, Acta Univ. Wratislav. No.
  2470):419--441, 2002.

\bibitem{MR1490808}
T.~Kulczycki.
\newblock Properties of {G}reen function of symmetric stable processes.
\newblock {\em Probab. Math. Statist.}, 17(2, Acta Univ. Wratislav. No.
  2029):339--364, 1997.

\bibitem{MR1643611}
T.~Kulczycki.
\newblock Intrinsic ultracontractivity for symmetric stable processes.
\newblock {\em Bull. Polish Acad. Sci. Math.}, 46(3):325--334, 1998.

\bibitem{MR2231884}
T.~Kulczycki and B.~Siudeja.
\newblock Intrinsic ultracontractivity of the {F}eynman-{K}ac semigroup for
  relativistic stable processes.
\newblock {\em Trans. Amer. Math. Soc.}, 358(11):5025--5057 (electronic), 2006.

\bibitem{2008-MK}
M.~Kwa{\'s}nicki.
\newblock Intrinsic ultracontractivity for stable semigroups on unbounded open
  sets.
\newblock preprint, 2008.

\bibitem{MR2280260}
P.~J. M{\'e}ndez-Hern{\'a}ndez.
\newblock Exit times of symmetric {$\alpha$}-stable processes from unbounded
  convex domains.
\newblock {\em Electron. J. Probab.}, 12:no. 4, 100--121 (electronic), 2007.

\bibitem{MR2213639}
K.~Michalik.
\newblock Sharp estimates of the {G}reen function, the {P}oisson kernel and the
  {M}artin kernel of cones for symmetric stable processes.
\newblock {\em Hiroshima Math. J.}, 36(1):1--21, 2006.

\bibitem{MR0208681}
S.~C. Port.
\newblock Hitting times for transient stable processes.
\newblock {\em Pacific J. Math.}, 21:161--165, 1967.

\bibitem{MR2238934}
M.~Rao, R.~Song, and Z.~Vondra{\v{c}}ek.
\newblock Green function estimates and {H}arnack inequality for subordinate
  {B}rownian motions.
\newblock {\em Potential Anal.}, 25(1):1--27, 2006.

\bibitem{MR1968386}
L.~Riahi.
\newblock Estimates of {G}reen functions and their applications for parabolic
  operators with singular potentials.
\newblock {\em Colloq. Math.}, 95(2):267--283, 2003.

\bibitem{MR1739520}
K.-i. Sato.
\newblock {\em L\'evy processes and infinitely divisible distributions},
  volume~68 of {\em Cambridge Studies in Advanced Mathematics}.
\newblock Cambridge University Press, Cambridge, 1999.
\newblock Translated from the 1990 Japanese original, Revised by the author.

\bibitem{MR2365348}
R.~L. Schilling and T.~Uemura.
\newblock On the {F}eller property of {D}irichlet forms generated by pseudo
  differential operators.
\newblock {\em Tohoku Math. J. (2)}, 59(3):401--422, 2007.

\bibitem{MR2255353}
B.~Siudeja.
\newblock Symmetric stable processes on unbounded domains.
\newblock {\em Potential Anal.}, 25(4):371--386, 2006.

\bibitem{MR0290095}
E.~M. Stein.
\newblock {\em Singular integrals and differentiability properties of
  functions}.
\newblock Princeton Mathematical Series, No. 30. Princeton University Press,
  Princeton, N.J., 1970.

\bibitem{MR1969798}
N.~T. Varopoulos.
\newblock Gaussian estimates in {L}ipschitz domains.
\newblock {\em Canad. J. Math.}, 55(2):401--431, 2003.

\bibitem{MR2286060}
T.~Watanabe.
\newblock Asymptotic estimates of multi-dimensional stable densities and their
  applications.
\newblock {\em Trans. Amer. Math. Soc.}, 359(6):2851--2879 (electronic), 2007.

\bibitem{MR1336382}
K.~Yosida.
\newblock {\em Functional analysis}.
\newblock Classics in Mathematics. Springer-Verlag, Berlin, 1995.
\newblock Reprint of the sixth (1980) edition.

\bibitem{MR1900329}
Q.~S. Zhang.
\newblock The boundary behavior of heat kernels of {D}irichlet {L}aplacians.
\newblock {\em J. Differential Equations}, 182(2):416--430, 2002.

\bibitem{MR1974093}
Q.~S. Zhang.
\newblock The global behavior of heat kernels in exterior domains.
\newblock {\em J. Funct. Anal.}, 200(1):160--176, 2003.

\bibitem{MR842803}
Z.~X. Zhao.
\newblock Green function for {S}chr\"odinger operator and conditioned
  {F}eynman-{K}ac gauge.
\newblock {\em J. Math. Anal. Appl.}, 116(2):309--334, 1986.

\end{thebibliography}
\def\cprime{$'$}

\noindent
Krzysztof Bogdan (Krzysztof.Bogdan@pwr.wroc.pl)

\noindent
Tomasz Grzywny (Tomasz.Grzywny@pwr.wroc.pl)

\noindent
Institute of Mathematics and Computer Science,

\noindent
Wroc{\l}aw University of Technology

\noindent
Wybrze{\.z}e Wyspia{\'n}skiego 27, 50-370, Wroc{\l}aw, Poland

\end{document}